\documentclass{IEEEojcsys}

\usepackage[colorlinks,urlcolor=blue,linkcolor=blue,citecolor=blue]{hyperref}
\usepackage{amsmath,amssymb,amsfonts}
 
\usepackage{amsthm}
\usepackage{graphicx}
\usepackage[dvipsnames]{xcolor}
\usepackage{cite}
\let\labelindent\relax
\usepackage{enumitem}
\usepackage{comment}
\usepackage{tabularx}
\usepackage{makecell}
\usepackage{multirow}
\usepackage{bm}
\usepackage{subcaption}
\usepackage{tikz}
\usetikzlibrary{automata,positioning,arrows,shapes,calc}

\newtheorem{definition}{Definition}
\newtheorem{lemma}{Lemma}
\newtheorem{theorem}{Theorem}
\newtheorem{proposition}{Proposition}

\newtheorem{axiom}{Axiom}

\DeclareMathOperator{\subjectto}{subject to}

\makeatletter
\renewcommand{\reviseddate}[1]{}
\makeatother

\jvol{}
\jnum{}
\paper{}
\pubyear{}
\receiveddate{}
\accepteddate{}
\publisheddate{}
\currentdate{}
\doiinfo{}

\begin{document}

\sptitle{Position Paper}

\title{Welfare and Cost Aggregation for Multi-Agent Control: When to Choose Which Social Cost Function, and Why?}

\author{Ilia Shilov* \affilmark{1}}
\author{Ezzat Elokda* \affilmark{1}}
\author{Sophie Hall* \affilmark{1}}
\author{Heinrich H. Nax \affilmark{2}}
\author{Saverio Bolognani \affilmark{1}}

\affil{Automatic Control Laboratory, ETH Zurich, 8092 Zurich, Switzerland}
\affil{Zurich Center for Market Design \& SUZ, University of Zurich, 8050 Zurich, Switzerland}

\authornote{* Equal contribution. \\
This work was supported by the NCCR Automation, a National Centre of Competence in Research, funded by the Swiss National Science Foundation (grant number 51NF40\textunderscore225155).}

\corresp{CORRESPONDING AUTHOR: Ilia Shilov (e-mail: ishilov@ethz.ch)}

\begin{abstract}
Many multi-agent socio-technical systems rely on aggregating heterogeneous agents’ costs into a social cost function (SCF) to coordinate resource allocation in domains like energy grids, water allocation, or traffic management. The choice of SCF often entails implicit assumptions and may lead to undesirable outcomes if not rigorously justified. In this paper, we demonstrate that what determines which SCF ought to be used is the degree to which individual costs can be compared across agents and which axioms the aggregation shall fulfill. Drawing on the results from social choice theory, we provide guidance on how this process can be used in control applications. We demonstrate which assumptions about interpersonal utility comparability - ranging from ordinal level comparability to full cardinal comparability - together with a choice of desirable axioms, inform the selection of a correct SCF, be it the classical utilitarian sum, the Nash SCF, or maximin.  Thus, fixing comparability level first, then choosing an objective from the compatible class, and reporting both as part of the specification, makes the fairness and efficiency consequences transparent. We demonstrate how the proposed framework can be applied for principled allocations of water, transportation, and energy resources. 

\end{abstract}


\maketitle

\section{Introduction}
\label{sec:intro}

Multi-agent socio-technical control applications with heterogeneous agents arise in various domains, such as energy grids, traffic control, water distribution, and bandwidth allocation~\cite{CSSroadmap}. In these systems, the control objective in the form of a social cost function (SCF) is typically context-specific, and it often involves some cost or utility aggregation across agents, reflecting agents’ different goals, needs, and operational constraints. When performance indicators (e.g., energy consumption, travel time, bandwidth usage) are measurable and objective, one might quite straightforwardly design an SCF such as total cost or throughput. However, when agents’ costs represent subjective, agent-specific valuations (e.g. electricity
needs of a hospital and a household, value of time for a population of travelers with heterogeneous income etc.), it is often unclear if these costs can be compared  on a common scale, leading to difficulties and ambiguities in aggregating these costs and defining SCFs. 

Several candidate SCFs, including the sum (or weighted sum) of costs, Nash social cost, and min-max objectives, exist in the literature \cite{caragiannis2019mnw, ramezani_endriss_2010, radunovic2007unified, bertimas2011priceoffairness}, with the choice among these criteria usually depending on desired properties such as tractability, fairness, and robustness. The most popular approach to aggregate individual costs is the \emph{classical utilitarian} rule, which sums individual costs into a single objective. While intuitive and computationally tractable, this approach assumes that the cost for one agent is commensurate with the cost of another. In practice, this is a strong comparability assumption which can lead to unintended consequences \cite{guide2023fairness}, e.g., disproportionally high wait times for ride-hails in remote areas~\cite{maciejewski2016assignment}, discrimination against certain train types or routes in real-time train re-scheduling~\cite{luan2017non}, or exacerbation of energy poverty~\cite{sousa2019peer}.

Exact \emph{full comparability} of the costs incurred by different agents can be impeded by various reasons:
\begin{itemize}
\item Costs may be subjective and not straightforward to compare (e.g., perceived value of travel time, which depends on socio-economic and circumstantial factors\cite{zamparini2007meta,liu2024adaptive}).
\item Costs may be difficult to quantify because they depend on complex models and unknown parameters (e.g., the yield of a farm as a function of the water allocated to it) or because they include information that is not available in the decision process.
\item The comparability framework is an elegant way to encode fairness considerations, i.e., what elements of agents' costs are deemed relevant to the decision problem at hand (e.g., a deliberate decision not to account for traffic delays accumulated outside the road network of interest).
\item When costs are self-reported, they are vulnerable to strategic manipulation (e.g., customers declaring their electricity needs \cite{muthirayan2019mechanism}).
\end{itemize}


\begin{figure}[t]
\centering
\includegraphics[width=0.9\columnwidth]{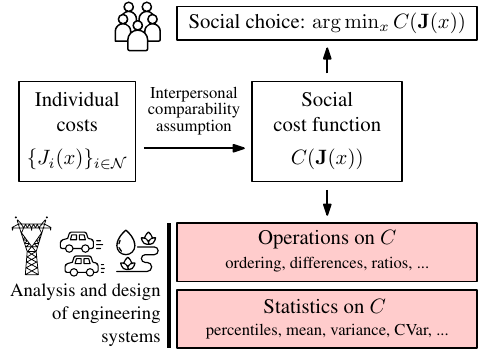}
\caption{Uses of social cost functions in the analysis and design of engineering systems.}
\label{fig:what-to-do-with-W}
\end{figure}

The central aim of this paper is to offer an axiomatic viewpoint on how to properly aggregate agents' costs in multi-agent decision and control depending on their comparability.
By reviewing, adapting, and extending the main concepts of \emph{welfarism} and \emph{interpersonal comparability} to the specific context of multi-agent control, 
this paper develops and recommends a design procedure that comprises the following steps (Figure~\ref{fig:what-to-do-with-W}):
\begin{enumerate}
    \item \textbf{State the level of interpersonal comparability}. 
    Such a decision is guided by how subjective agents' preferences are and how well they can be measured, but it can also be a deliberate decision guided by politics, social agreements, and public sentiment.

    \item \textbf{Select the appropriate social cost function}.
    Based on the determined comparability level, the choice of SCF is  restricted to a specific functional form.

    \item \textbf{Align the control design to the social cost function}.
     The design of the sociotechnical system must be guided uniquely by the explicit optimization of the selected SCF. 
     Criteria that cannot be expressed through the SCF (e.g., rights, fairness constraints, external moral values, distributional constraints) should not be employed.
    %
     
\end{enumerate}
We urge control engineers to adopt this rigorous and explainable process in all application domains where the design is guided by social cost/welfare functions \cite{Munoz2023}.
The mathematical guarantees that we provide can be incorporated in the design of system-wide metrics that are now being used in multi-agent control, like Price of Anarchy \cite{Chandan2024,Zhang2018,Piliouras2017,Wang2017,hill2023tradeoff}.
They can also guide the design of optimization dynamics \cite{jadbabaie2009distributed,Menon2014,romvary2021proximal,Camisa2020OptNash} and resource allocation mechanisms \cite{Farhadi2019,Maheswaran2004}.
Finally, this work responds to the growing interest in fairness in control~\cite{CSSroadmap,jalota2021efficiency,villa2023fair,bang2024mobility,elokda2024carma,khargonekar2024climate,elokda2025vision}, as many fairness notions arise naturally from careful considerations of the level of interpersonal comparability~\cite{elokda2025vision}. 
This approach can also clarify the implicit fairness assumptions in existing game-theoretic solution concepts, such as the variational Generalized Nash Equilibrium, which has been shown to rely on a strong, and often unstated, comparability assumption that, if violated, can lead to undesirable ``fairness'' outcomes \cite{hall2025limits}.

The aforementioned three-step procedure can also be interpreted in reverse: if a specific notion of efficiency/fairness is desired (utilitarian, min/max, etc.) then a sufficient level of comparability needs to be achieved, which may require the collection of additional information from the agents, the design of a manipulation-safe mechanism, or the use of more accurate cost models.
This may give a principled perspective on how freely agents' costs can be designed and engineered \cite{Marden2013,Jensen2018, Marden2014}.

The third step of the procedure consists of using the selected SCF for control design. 
These SCFs are functionally similar to some objectives used in multi-objective optimization \cite{Liu2001MOOOptandControl,Chinchuluun2007MOOsurvey}, and specifically in multi-objective control design problems \cite{Gambier2007MOOoverview,Gambier2008MPCandPIDMOO}.
For example, linear-quadratic regulators can be designed to minimize weighted-sum \cite{DeNicolao1993UtopianLQ}, 
minmax \cite{Li1990MinmaxLQ}, and proportional \cite{Shtessel1996proportionalLQR} costs.
Similar cost functions can be incorporated in receding horizon control decisions
\cite{DeVito2007Recedinghorizon,Bemporad2009MultiobjectiveMPC,Zavala2012Utopiatracking} and in computational control design tools 
\cite{Reynoso2017controllerTuning}.
While these works do not address the aggregation of costs across multiple agents, they offer methods and algorithms for implementing social choice in practice.

The remainder of this paper is structured as follows. 
In Section \ref{sec:welfarism}, we set the stage for a principled \emph{welfarist} approach based on social choice theory. 
In Section \ref{sec:aggregate-utilities} we show how preferences can be properly aggregated into a SCF based on their level of interpersonal comparability. 
In Section~\ref{sec:examples} we illustrate the proposed design process via three examples of sociotechnical systems.
Section~\ref{sec:conclusions} concludes the paper.
\section{Welfarism}
\label{sec:welfarism}

In this section we introduce preliminaries adopted from social choice theory \cite{DAspremont1977, roberts_interpersonal_1980, dAspremontGevers2002} that allow a rigorous aggregation of individual costs and derivation of social cost functions. We lay out the foundations of a so called ``welfarist'' approach, that requires that all the relevant information for the social decision is contained in the agents' cost functions.

\subsection{Preliminaries and Axiomatic Foundations}
\label{sec:notation}
\label{subsec:agents}

Let $\mathcal{N} = \{1,2,\dots,n\}$ be a finite set of agents and let $\mathcal{X}$ denote the set of feasible outcomes $x$.
An outcome can be a specific allocation of a scarce resource or any control decision that affects the agents. Each agent $i$ attaches a (possibly negative) \emph{cost} $J_i : \mathcal X \to \mathbb R$ to every outcome, with $J_i(\cdot) \in \mathcal{J}$, where $\mathcal{J}$ is the set of all real-valued functions on $\mathcal X$.
We let $\mathbf{J}=(J_1,\dots,J_n)$ denote a \emph{profile} of such costs for the entire set of agents. When formulating a resource allocation problem that depends on agents' individual evaluations represented by $\mathbf{J}$, one needs to define a single \emph{social preference relation} $\succsim_{\mathbf{J}}$ on  $\mathcal{X}$, so as to capture the collective or ``social'' viewpoint. We first define a mapping between cost profiles of the agents and the resulting social preference, which can be expressed using Social Cost Functionals (SCFL):

\begin{definition}
\label{def:SWFL}
Let $\mathcal{R}$ be the set of all complete and transitive binary relations on $\mathcal{X}$.
A \emph{Social Cost Functional (SCFL)} is a mapping
\begin{equation*}
    \mathfrak F \;:\; \mathcal{J}^n \;\longrightarrow\; \mathcal{R},
\end{equation*}
that assigns a \emph{social preference relation} $\succsim_{\mathbf{J}} = \mathfrak{F}(\mathbf{J})$ on $\mathcal{X}$ to each profile $\mathbf{J}=(J_1,\dots,J_n)\in \mathcal{J}^n$.
\end{definition}

Before determining how to incorporate each agent's evaluation into a collective decision, one should first specify which \emph{fundamental properties (or axioms)} this social preference relation $\succsim_{\mathbf{J}}$ needs to satisfy. The following two classical well-established properties can be expected from such a social preference relation.

\begin{axiom}[Weak Pareto Principle \textbf{(P)}]
For any profile $\mathbf{J}$ and any $x,y\in \mathcal{X}$, 
if $J_i(x) < J_i(y)\text{ for all }i$ 
then 
$x \succ_{\mathbf{J}} y$. 
\end{axiom}

\begin{axiom}[Independence of Irrelevant Alternatives \textbf{(IIA)}]
For any two distinct outcomes $x,y\in\mathcal{X}$ 
and any two profiles $\mathbf{J},\mathbf{J}'\in\mathcal{J}^n$ such that 
$J_i(x)=J_i'(x)$ and $J_i(y)=J_i'(y)$ for each $i$, 
we require
\begin{equation*}
    x \succsim_{\mathbf{J}} y 
    \;\;\Longleftrightarrow\;\;
    x \succsim_{\mathbf{J}'} y,
    \quad
    \text{and}
    \quad
    x \succ_{\mathbf{J}} y 
    \;\;\Longleftrightarrow\;\;
    x \succ_{\mathbf{J}'} y.
\end{equation*}
In other words, changing the costs of \emph{other} outcomes does not affect the pairwise social ranking of $x$ and $y$. 
\end{axiom}

In control applications, it is convenient to work with real-valued measures of social cost rather than a social ordering of outcomes. This requires an additional mild continuity assumption, which postulates that if an outcome $x$ is socially strictly preferred over $y$, it should remain so under a small enough perturbation of the individual costs. As we show in Section~\ref{subsec:SWF-classes}, this condition is easily verified for the setting considered in this work.

\begin{definition}[Pairwise Continuity \textbf{(PC)} \cite{hammond2023roberts}]
    Let $\mathbb{R}_{++}$ denote the set of strictly positive real numbers. For every $\varepsilon \in \mathbb{R}_{++}^N$, there exists $\varepsilon' \in \mathbb{R}_{++}^N$ such that for every profile $\mathbf{J}$ and every pair $x,y \in \mathcal X$ with $x\,\succ_{\mathbf{J}}\,y$, there exists a profile $\tilde{\mathbf{J}}$ satisfying $\tilde{\mathbf{J}}(x) \geq \mathbf{J}(x) + \varepsilon'$ and $\tilde{\mathbf{J}}(y) \leq \mathbf{J}(y) + \varepsilon$, and such that $x\,\succ_{\tilde{\mathbf{J}}}\,y.$
\end{definition}

In some cases it may be desirable to enforce an additional axiom that prohibits distinguishing individual agents.

\begin{axiom}[Anonimity \textbf{(A)}]
Let \(\pi: \mathbb{N} \to \mathbb{N}\) be a permutation (i.e., a bijection on \(\mathbb{N}\)). 
If, for all \(i \in \mathbb{N}\) and \(x \in X\), $J'_i(x) = J_{\pi(i)}(x)$ then $\succsim_{\mathbf{J}} = \succsim_{\mathbf{J}'}$.
\end{axiom}

\subsection{Welfarism}

If we accept these axioms \textbf{(P)} and \textbf{(IIA)} as reasonable and impose them on our social preference relation, together with condition \textbf{(PC)}, we can establish a \emph{welfarist principle} \cite{Sen1979_welfarism,DAspremont1977}, that implies that any complete and transitive social ordering can be expressed through a \emph{Social Cost Function (SCF)} $C$ that depends only on the preference profile $\mathbf{J}$.

\begin{proposition}[Welfarism; Thm 3.7, \cite{dAspremontGevers2002}, Thm 1, \cite{roberts_interpersonal_1980}]
\label{pro:welfarism}
Let $\succsim_{\mathbf{J}}$ be a social preference relation on $\mathcal{X}$ 
defined by SCFL $\mathfrak{F}$ for any profile $\mathbf{J}=(J_1,\dots,J_n) \in \mathcal{J}^n$. Suppose $\succsim_{\mathbf{J}}$ satisfies \textbf{(P)}, \textbf{(IIA)} and \textbf{(PC)}. Then there exists a continuous Social Cost Function (SCF) 
\begin{equation*}
    C : \mathbb{R}^n \,\to\, \mathbb{R}
\end{equation*}
such that for any $x,y\in\mathcal{X}$, 
\begin{equation*}
    x \succsim_{\mathbf{J}} y \,\, \Leftrightarrow \,\, C(J_1(x),\dots,J_n(x)) \,\,  \leq \,\, C(J_1(y),\dots,J_n(y)).
\end{equation*}
\end{proposition}

The SCF $C$ thus represents the SCFL $\mathfrak{F}$, and \emph{all} relevant information for ranking outcomes is contained in the vector $\mathbf J(x) = \bigl(J_1(x),\dots,J_n(x)\bigr)$ of agents' costs for each outcome $x$.
The \emph{welfarist} approach prescribes that the social choice must be guided exclusively by the social cost function $C(\mathbf J(x))$, rather than by any other characteristics of the decision $x$.

As mentioned in the introduction, several well-known examples of SCFs $C$ have been proposed to aggregate individual costs. For instance, the \emph{utilitarian} approach, often attributed to Bentham \cite{bentham_introduction_1789} as later formalized by Harsanyi \cite{harsanyi_cardinal_1955}, aggregates costs via a sum: $C_{\mathrm{util}}(J_1,\dots, J_n)=\sum_{i=1}^n J_i$. By contrast, the \emph{Rawlsian} or \emph{min-max} rule, inspired by Rawls \cite{rawls_theory_1971}, focuses on the well-being of the worst-off agent: $C_{\mathrm{Rawls}}=\max_{i\in\mathcal{N}} J_i$. Another example is the \emph{Nash Social Welfare} function \cite{nash_bargaining_1950, kaneko1979nash}, which is defined as $C_{\mathrm{Nash}}(J_1,\dots,J_n)=-\prod_{i=1}^n (-J_i)$. Notably, these three canonical SCFs can all be viewed as members of the family of Hölder means. 
Generally, a large body of literature guides the selection of a SCF based on desirable notions of fairness, equity (e.g. envy-freeness and its relaxations, proportional fairness, Pigou-Dalton principle etc.), efficiency, or some trade-off between the two,  reflected in measures such as maximin share (MMS) guarantee, alpha-fairness, or statistical inequality indices (e.g., the Gini coefficient, coefficient of variation, or Hoover index) \cite{moulin_fair_2003, bertimas2011priceoffairness, guide2023fairness, bogomolnaia_competitive_2017}. Instead, the focus of this work is to provide a ``first-principles'' way to make this choice by first considering the fundamental assumptions of \emph{interpersonal comparability} of the agents' costs, which are needed to properly define ``efficiency'' and ``fairness'' in the first place, and from which the appropriate aggregation rule naturally follows.


\section{Cost Aggregation and Comparability}
\label{sec:aggregate-utilities}

The \emph{welfarist} principle (Proposition~\ref{pro:welfarism}) tells us 
that any social ordering under \textbf{(P)}, \textbf{(IIA)} and \textbf{(PC)} must be representable by a SCF that depends solely on the costs incurred by the agents. One must still decide \emph{which} $C$ is appropriate in practice. Such a choice needs to be guided by a deliberate decision on \emph{how} individual costs can be \emph{measured and compared} across agents, so that meaningful interpersonal trade-offs can be made in the social choice. Addressing the \emph{interpersonal comparability} issue corresponds to deciding which transformations of the individual costs leave the social ranking unchanged \cite{DAspremont1977,roberts_interpersonal_1980}. By restricting or expanding the class of \emph{invariance transformations} of the individual costs, one effectively selects a level of  \emph{comparability} for the agents' costs. Once the comparability assumptions are fixed, the shape of the SCF $C$ is essentially pinned down. 


\subsection{Comparability Levels}
\label{subsec:E-comparability}

\begin{definition}[Invariance Transformations under a SCFL]
Let $\mathfrak{F} : \mathcal{J}^n \to \mathcal{R}$ 
be a social cost functional. 
A vector of transformations $\{\phi_i \in \Phi : i \in \mathcal{N}\}$ is called an \emph{invariance transformation} under $\mathfrak{F}$ if, for any cost profile $\mathbf{J}=(J_1,\dots,J_n)$, the transformed profile
$$
\mathbf{J}' = (\phi_1 \circ J_1,\dots,\phi_n \circ J_n)
$$
is equivalent to $\mathbf{J}$ under $\mathfrak{F}$, meaning that for all $x,y \in \mathcal{X}$,
$$
x \succsim_{\mathbf{J}} y \quad \Longleftrightarrow \quad
x \succsim_{\mathbf{J}'} y.
$$
\end{definition}

We present four \emph{levels of interpersonal comparability}, defined via the corresponding invariance transformations.
Notice that by imposing a specific invariance condition, we are equivalently defining an information filter that specifies what features of the individual cost profiles are relevant for the social ordering. 
For each comparability level, we comment on what kind of equivalence they induce on the costs $J_i$ in terms of comparability.

\begin{description}

\item[Ordinal Level Comparability (OLC).]
The costs  are determined up to any \emph{common} strictly increasing transformation $\phi_{\text{\tiny OLC}}$:
\begin{equation*}
    J_i'(x) = \phi_{\text{\tiny OLC}}\bigl(J_i(x)\bigr)
\end{equation*} 
We cannot compare cost increments, but it is possible to order the costs incurred by different agents.

\smallskip

\item[Cardinal Non-Comparability (CNC).]
The costs are determined up to \emph{distinct} positive affine transformations:
\begin{equation*}
    J_i'(x) = \phi_{\text{\tiny CNC}}(J_i(x)) :=  a_i\,J_i(x) \;+\; b_i,
\quad a_i>0,
\end{equation*}
i.e., each agent has its own scale and origin. We can compare increments \emph{within} a single agent’s cost, but we cannot compare costs or increments across agents.

\item[Cardinal Unit Comparability (CUC).]
The costs are determined up to any \emph{common} scale factor $a$, but may have \emph{distinct} offsets $b_i$:
\begin{equation*}
    J_i'(x) = \phi_{\text{\tiny CUC}} (J_i(x)) :=  a\,J_i(x) \;+\; b_i, \quad a>0,
\end{equation*}
We can compare increments in cost across agents; however, absolute costs across agents are not comparable.

\smallskip

\item[Cardinal Full Comparability (CFC).]
The costs are determined up to \emph{common} positive affine transformations:
\begin{equation*}
    J_i'(x) = \phi_{\text{\tiny CFC}} (J_i(x)) := a\,J_i(x) \;+\; b, \quad a>0.
\end{equation*}
All agents share a single absolute scale, so both cost levels \emph{and} increments are comparable.

\end{description}

\begin{figure}[tb!]
    \centering
    \includegraphics[width=0.6\columnwidth]{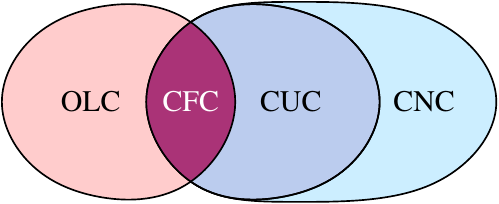}
    \caption{Venn diagram of the interpersonal comparability levels. Cardinal Full Comparability (CFC) is often unwittingly assumed, with unintended fairness consequences.}
    \label{fig:venn}
\end{figure}

Figure~\ref{fig:venn} illustrates the hierarchy of these comparability levels.
One should carefully consider what comparability level and associated invariance condition are justified in a given socio-technical control setting.
Too little comparability may prevent meaningful trade-offs, whereas too much comparability might overstate the legitimacy of comparisons across agents whose costs are inherently heterogeneous.

\subsection{Permissible Social Cost Functions}
\label{subsec:SWF-classes}

\def\iconscale{0.6}
\newlength{\iconsize}
\setlength{\iconsize}{67pt}

\begin{table*}[bt]
    \centering
    \renewcommand{\arraystretch}{2}
    \setcellgapes{3pt}
    \makegapedcells
    \begin{tabular}{|c|c|c|c|c|}
    \hline
        \textbf{Comparability} & OLC & CNC & CUC & CFC \\
        \hline
        \textbf{Invariant tf.} & incr. $\phi_\text{OLC}(J_i)$ 
         & $\phi_{\text{\tiny CNC}}(J_i) = a_i J_i + b_i$
        & $\phi_{\text{\tiny CUC}}(J_i) = a J_i + b_i$
        & $\phi_{\text{\tiny CFC}}(J_i) = a J_i + b$ \\
        \hline
        \textbf{SCF $C(x)$} 
         & $\max\limits_{i} J_i(x)$
         & \begin{tabular}{c}
            \makecell{
            $- \prod_{i} \left[J_i(x_0) - J_i(x)\right]^{c_i}$\\[1mm]
            \footnotesize with benchmark $x_0$}\\
            \makecell{$-\prod_{i} \left[ -J_i(x) \right]^{c_i}$\\[1mm]
            \footnotesize with $b_i=0$ \\and negative costs}
         \end{tabular}
         & $\sum\limits_{i} c_i\,J_i(x)$
         & \makecell{$\frac{1}{n}\sum\limits_{i} J_i(x)+ g\left( 
\left[\begin{smallmatrix}
    J_1(x) - \frac{1}{n}\sum\limits_{i} {J_i}(x) \\
    \vdots \\
    J_n(x) - \frac{1}{n}\sum\limits_{i} {J_i}(x)
\end{smallmatrix}\right]
\right)$\\[6mm]
         $g$ homogeneous of degree 1}
         \\
         \hline
         \begin{minipage}[b]{18mm}
            \centering
            \bf Allowed operations on the SCF $C(x)$\\
            ~
         \end{minipage} 
         & 
         \includegraphics[scale=\iconscale]{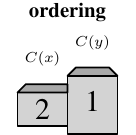} & 
          \includegraphics[scale=\iconscale]{figures/op-ordering.pdf}\hspace{3mm}
            \includegraphics[scale=\iconscale]{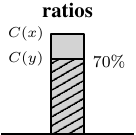}&
         \multicolumn{2}{c|}{
            \includegraphics[scale=\iconscale]{figures/op-ordering.pdf}\hspace{3mm}
            \includegraphics[scale=\iconscale]{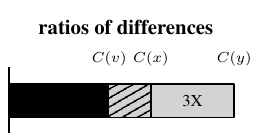}
         }\\
         \hline 
    \end{tabular}
    \medskip 
    \caption{Synoptic table representing the appropriate SCF and SCF operations for each comparability level (Theorem~\ref{thm:swf-classification}).}
    \label{table}
\end{table*}

Classical results in the literature — most notably by Sen \cite{sen_interpersonal_1970,Sen1979_welfarism}, d’Aspremont and Gevers \cite{dAspremontGevers2002}, and Roberts \cite{roberts_interpersonal_1980} — demonstrate that if a Social Cost Functional (SCFL) satisfies \textbf{(P)} and \textbf{(IIA)}, then its form is uniquely determined by the level of interpersonal comparability assumed. We first provide a technical result informally stated in \cite{hammond2023roberts}, then formalize the SCF choice in  Theorem~\ref{thm:swf-classification}, summarizing results from the literature above in the formalism of our paper (see Figure~\ref{fig:tikzdiagram} for a visual guide and Table~\ref{table} for a synoptic table).

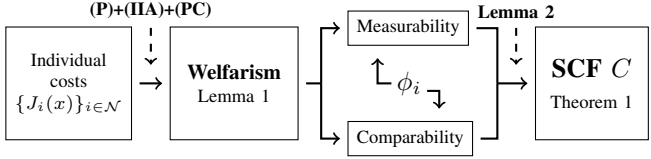
\begin{figure}[b]
  \centering
  \begin{tikzpicture}[node distance=2cm, auto]
    \tikzset{myarrow/.style={->, thick, shorten >=1mm, shorten <=1mm}}
    
    \node[draw, rectangle, align=center, minimum width=1.5cm, minimum height=1.8cm, font=\footnotesize] (A) 
      {\textbf{Individual} \\
      \textbf{costs} \\[1mm]
      $\{J_i(x)\}_{i \in \mathcal{N}}$};
    
    \node[draw, rectangle, align=center, right=2.2cm of A, minimum width=1.8cm, minimum height=1.8cm, font=\footnotesize] (B)
      {\textbf{Welfarism:} \\ 
      \textbf{SCF}\\$C(\mathbf{J}(x))$\\
      {\scriptsize Proposition \ref{pro:welfarism}}};
    
    \draw[myarrow] (A) -- (B);
    
    \path (A) -- (B) coordinate[midway] (midpointAB);
    \coordinate (midAB) at ($(midpointAB)+(2mm,0)$);

    \node[above=5mm of midAB, font=\scriptsize, anchor=south] (labelAB) {\textbf{(P) + (IIA)}};

    \node[draw, rectangle, align = center, right=0.7cm of B, minimum width=1.8cm, minimum height=1.5cm,font=\footnotesize] (C) {\textbf{Choice of}\\\textbf{SCF} $C$ \\ {\scriptsize Theorem \ref{thm:swf-classification}}};

    \node[draw, rectangle, minimum width=24mm, text centered, below=0.2cm of B] (D) {\footnotesize \textbf{Comparability}};

    \node[align=right,below right=-3mm and -7mm of D.west, font=\scriptsize, anchor=east] (labelPC) {\textbf{(PC)}\\Lemma 1};
    
    \draw[myarrow, dashed] (D) -| (midAB);
    
    \draw[myarrow, dashed] (D) -| (C);

    \draw[myarrow, dashed] (labelAB.south) -- (midAB.south);

    \draw[myarrow] (B) -- (C);

    \node[below=1mm of D.south] (spacer) {};
\end{tikzpicture}
  \caption{Guide to the results of Sections~\ref{sec:welfarism} and \ref{sec:aggregate-utilities}.}
  \label{fig:tikzdiagram}
\end{figure}

\begin{lemma}\label{lemma:PC}
    \textbf{(OLC), (CNC), (CUC), (CFC)} imply Pairwise Continuity \textbf{(PC)}.
\end{lemma}
\begin{proof}
Let $\varepsilon=(\varepsilon_1,\ldots,\varepsilon_n)\in\mathbb{R}_{++}^n$ and set $\alpha=\min_i \varepsilon_i$. 
Define a common shift $\tilde J_i(x)=J_i(x)+\alpha$ for all $i$ and all $x\in\mathcal X$. 
Invariance under common additive shifts yields $x\succ_{\mathbf J} y \Rightarrow x\succ_{\tilde{\mathbf J}} y$ for every pair $x,y\in\mathcal X$. 
Let $\varepsilon'=\alpha \mathbf 1$. Then, component-wise, $\tilde{\mathbf J}(x)\ge \mathbf J(x)+\varepsilon'$ and $\tilde{\mathbf J}(y)=\mathbf J(y)+\alpha\le \mathbf J(y)+\varepsilon$. So the condition \textbf{(PC)} holds with $\varepsilon'$, uniformly across profiles and pairs.
\end{proof}

\begin{theorem}[Choice of Social Cost Functions \cite{roberts_interpersonal_1980}, \cite{dAspremontGevers2002}]
\label{thm:swf-classification}
Let $\mathfrak{F} : \mathcal{J}^n \to \mathcal{R}$ be a Social Cost Functional satisfying \textbf{(IIA)} $+$ \textbf{(P)} (unless stated otherwise). We have the following results, for each level of interpersonal comparability.

\begin{description}

\item[(OLC):]  The SCFL $\mathfrak{F}$ is represented\footnote{An additional mild equity argument is needed, see~\cite{dAspremontGevers2002}.} by the SCF
\[
C(\mathbf{J}(x)) \;=\; \max_{i\in\mathcal{N}} J_i(x).
\]

\item[(CNC):]
    A SCFL that satisfies the axioms does not exist. However, if we assume Partial Independence\footnote{There exists a fixed reference outcome $x_0 \in \mathcal{X}$ such that for any two cost profiles $\mathbf{J}, \mathbf{J}' \in \mathcal{U}^n$ and for any subset $A \subset \mathcal{X}$, if $J_i(x)=J'_i(x) \quad \text{for all } x\in A \cup \{x_0\} \text{ and for all } i\in\mathcal{N}$, then the social preference relations induced by $\mathbf{J}$ and $\mathbf{J}'$ coincide on $A$, i.e., $\forall\,x,y \in A:\quad x \succsim_{\mathbf{J}} y \;\Longleftrightarrow\; x \succsim_{\mathbf{J}'} y.$} \textbf{(PI)} instead of \textbf{(IIA)}, then the SCFL $\mathfrak{F}$  is represented by the SCF
       \[
           C(\mathbf{J}(x)) \;=\; -\prod_{i\in\mathcal{N}} \Bigl[J_i(x_0) - J_i(x)\Bigr]^{c_i}, \qquad c_i>0.
       \]
       where $x_0$ is a fixed benchmark outcome such that $J_i(x_0) \ge J_i(x)$ for all $x$.

       \smallskip
\item[(CNC with $b_i=0$):]
       If all costs $J_i$ are strictly negative then the SCFL $\mathfrak{F}$ is represented by the SCF
        \[
            C(\mathbf{J}(x)) \;=\; -\prod_{i\in\mathcal{N}} \Bigl[\,-J_i(x)\Bigr]^{c_i}, \qquad c_i>0.
        \]
        
\item[(CUC):] 
The SCFL $\mathfrak{F}$ is represented by the SCF
\[
C(\mathbf{J}(x)) \;=\; \sum_{i\in\mathcal{N}} c_i\,J_i(x), \qquad c_i>0.
\]

\item[(CFC):] The SCFL $\mathfrak{F}$ is represented by the SCF

\begin{equation*}
\!C(\mathbf{J}(x)) = \frac{1}{n}\sum_{i\in\mathcal N} J_i(x)\;+\; g\left( 
\left[\begin{smallmatrix}
    J_1(x) - \frac{1}{n}\sum_{i\in\mathcal N} {J}(x) \\
    \vdots \\
    J_n(x) - \frac{1}{n}\sum_{i\in\mathcal N} {J}(x)
\end{smallmatrix}\right]
\right)
\end{equation*}
    with $g: \mathbb R^n \to \mathbb R$ homogeneous of degree 1.\footnote{$g$ is homogeneous of degree 1 if $g(\lambda x)=\lambda\,g(x)$ for all $\lambda>0$.} For example one could choose $\gamma \max(\cdot)$ with $0 \leq \gamma \leq 1$, thus reflecting a balance between efficiency and equity, nicely axiomatized in \cite{bossert}.

\end{description}

Additionally, if one requires the axiom of Anonymity \textbf{(A)}, then all the $c_i$ in \textbf{(CNC)} and \textbf{(CUC)} become equal, i.e. $c_i = c_j, \forall i,j \in \mathcal{N}$, and the function $g$ in \textbf{(CFC)} must be invariant to permutations of the agents.
\end{theorem}

\subsection{Admissible Operations on Social Cost Functions}
\label{ssec:usesofW}

In many engineering problems, one is interested in using the social cost function to guide other operations rather than minimizing it to select the optimal outcome (Figure~\ref{fig:what-to-do-with-W}).
For example, it may be necessary to perform a quantitative comparison between two possible outcomes (that is, not necessarily selecting the optimal one, but also assessing relative performances). 
Moreover, in stochastic settings, it may be necessary to compute some statistics on the value of the SCF, like estimating confidence intervals and percentiles, means, variance, or other quantities useful in risk assessment and stochastic optimization.

One needs to be particularly careful about using SCFs for these purposes because Proposition~\ref{pro:welfarism} only certifies their use for selecting the best social choice. It is allowed, however, to perform other operations on the values of SCF for different outcomes, as long as the result of such operations is invariant with respect to the transformations listed in Section~\ref{subsec:E-comparability}, which guarantees that the result of such operations is meaningful.
This is stated in the following results.
\begin{lemma}\label{lemma:invariance}
    Consider $n$ values $(w_1,\dots,w_n) \in \mathbb{R}^n$. Suppose $q(w_1,\dots,w_n)$ is a rational function satisfying, $\forall\,a \in \mathbb{R}_{++},\;\forall\,b \in \mathbb{R}$,
    \begin{equation*}
        q(w_1,\dots,w_n) 
        \;=\;
        q(a\,w_1 + b,\;\dots,\;a\,w_n + b)    
    \end{equation*}
    Then $q$ is a rational function of the \emph{ratios of differences}:
    \begin{equation*}
        \frac{\,w_i - w_j\,}{\,w_k - w_\ell\,}, 
        \quad
        1 \le i,j,k,\ell \le n, 
        \quad (w_k - w_\ell \neq 0).  
    \end{equation*}
    Conversely, any rational expression built from these ratios of differences is invariant under 
    $t \mapsto a\,t + b$.
\end{lemma}

\begin{proof}
As $q$ is a rational function, w.l.o.g we check the identity on the open region where $w_1 > w_n$. In this region, choose the parameters $a = \frac{1}{w_1 - w_n} > 0, \quad b = -a w_n.$ Substituting these into the invariance condition yields
\[
q(w_1,\dots,w_n)
\;=\;
q\Bigl(
1,
\frac{w_2 - w_n}{w_1 - w_n},
\dots,
\frac{w_{n-1} - w_n}{w_1 - w_n},
0
\Bigr).
\]
The right-hand side is a rational expression depending solely on the ratios $\frac{w_k - w_n}{w_1 - w_n}$ (or, equivalently, any ratio of differences $\frac{w_i - w_j}{w_k - w_\ell}$). This holds on the entire domain, as the equality holds on an open set and $q$ is a rational function. The converse is immediate.
\end{proof}

\begin{proposition}[Invariance of ratios of SCF differences]\label{proposition: invariance SCF}
    Let $\mathcal X' \subset \mathcal X$ be a finite set of outcomes $x$ and $\mathbf{J}$  a profile of cost functions. 
    Then, under \textbf{(CUC)} or \textbf{(CFC)}, a rational function of the admissible SCF evaluations $\left\{C(\mathbf{J}(x))\right\}_{x \in \mathcal X'}$ is invariant under the corresponding transformations if and only if it is a function of the ratios of differences 
    \begin{equation}
    \frac{C(\mathbf{J}(x)) - C(\mathbf{J}(z))}{C(\mathbf{J}(y)) - C(\mathbf{J}(w))}, \quad \text{with }x,y,z,w \in \mathcal X',
    \label{eq:admissibleoperations}
    \end{equation}
    or, in case of \textbf{(CNC)}, of the ratios $\dfrac{C(\mathbf{J}(x), \mathbf{J}(x_0))}{C(\mathbf{J}(y), \mathbf{J}(x_0))}, \ x,y \in \mathcal X'$ and $x_0$ being the reference point. 
\end{proposition}

\begin{proof}
    We show that admissible transformations described in Section~\ref{subsec:E-comparability} induce an affine group action on the SCF $C$. 
    For \textbf{(CUC)} we obtain
    $C(\phi_\text{CUC}(\cdot)) = a C(\cdot) + \sum_{i} c_i b_i$.
    For \textbf{(CFC)} we obtain
    $C(\phi_\text{CFC}(\cdot)) = a C(\cdot) + n b$.
    For \textbf{(CNC)} we obtain
    $C(\phi_\text{CNC}(\cdot)) = (\prod_i a_i^{c_i}) C(\cdot). $
    Thus, Lemma \ref{lemma:invariance} can be applied to characterize the full set of rational functions that are invariant under such transformation, with the \textbf{(CNC)} reduced to ratios by picking $w_j = w_k$ and then dividing the ratio-of-differences by $w_k$.
\end{proof}
\smallskip


It should not be surprising that other operations on the SCF $C$ are not generally permissible: $C$ is a cardinal representation of the order imposed by the SCFL $\mathfrak F$. 
Therefore, additional measures, such as setting a fixed reference point or a fixed scale, are needed to make other operations (e.g., absolute differences or ratios) meaningful.
For instance, by selecting $z=w=v$ in \eqref{eq:admissibleoperations}, the ratios-of-differences takes the form 
\begin{equation}
t := \frac{C(\mathbf{J}(x)) - C(\mathbf{J}(v))}{C(\mathbf{J}(y)) - C(\mathbf{J}(v))},
\label{eq:ratioOfImprovements}
\end{equation}
which has an immediate interpretation: going from allocation $v$ to allocation $x$ is $t$ times better/worse than going from allocation $v$ to allocation $y$.

As an example, consider the use of Price of Anarchy (PoA) to evaluate and optimize system performance \cite{Chandan2024,Zhang2018,Piliouras2017,Wang2017,hill2023tradeoff}.
PoA quantifies (in)efficiency of a mechanism via the ratio between the value that the SCF takes at the worst-case Nash equilibrium and at the social optimum. However, in its classical form, PoA is not invariant to the affine transformations and one needs to replace it with a metric that is based on the admissible operations defined in Proposition \ref{proposition: invariance SCF}. 
One can adopt the ratio between the improvement of the Nash equilibrium and the improvement of the social optimum, both with respect to a benchmark solution (or, in the case of CNC, with respect to the reference point).

\section{Examples}
\label{sec:examples}

In this section, we illustrate how the methodology proposed in this paper is applicable to engineering problems comprising multiple agents competing for a limited resource.
We considered three timely domains: water allocation, traffic control, and curtailment of renewable energy.
We illustrate how a designer can decide the appropriate interpersonal comparability level, considering the available information and societal or political perceptions, and select an appropriate social cost function.

\subsection{Water Allocation}
\label{subsec:water}

Agricultural irrigation accounts for 70\% of global freshwater use already, and projections are that irrigation-based food production will need to grow another 50\% by 2050 due to climate change in combination with population growth~\cite{bwambale2022smart}. As groundwater reservoirs deplete~\cite{condon2021global}, water will need to be used more efficiently and prudently to avoid resource collapse~\cite{li2017irrigation}, and allocation/priority rules are needed to avoid a ``tragedy of the commons''~\cite{olorunfemi2023how}. 
Recent works have explored advanced control techniques for water allocation~\cite{negenborn2009distributed, castelletti2023model,valledesma2024water,wang2025stochastic}.

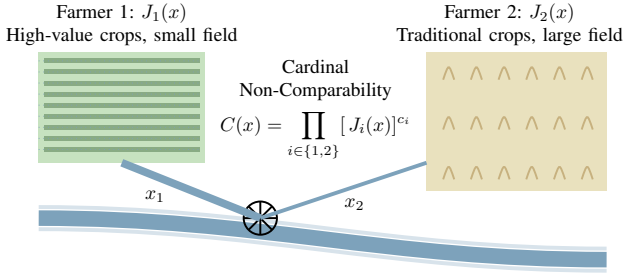
\begin{figure}[hbt!]
    \centering
\resizebox{0.98\columnwidth}{!}{
\begin{tikzpicture} 
\definecolor{waterblue}{RGB}{125, 162, 187}
\definecolor{fieldgreen}{RGB}{199, 226, 192}
\definecolor{cropgreen}{RGB}{137, 171, 132}
\definecolor{wheatyellow}{RGB}{233, 225, 190}
\definecolor{wheatgold}{RGB}{199, 177, 127}

\begin{scope}[shift={(0,0)}]
\fill[fieldgreen] (1,0) rectangle (4,2);
\foreach \y in {0.2,0.4,0.6,0.8,1.0,1.2,1.4,1.6,1.8} {
  \fill[cropgreen] (1.1,\y) rectangle (3.9,\y+0.08);
  \draw[waterblue, dotted] (1,\y+0.04) -- (3.9,\y+0.04);
}
\end{scope}

\begin{scope}[shift={(7,0)}]
\fill[wheatyellow] (1,-0.5) rectangle (4.25,2);
\foreach \x in {1.3,1.8,...,4} {
  \foreach \y in {-0.3,0.6,...,1.75} {
    \draw[wheatgold, line width=1pt] ($(\x,\y)$) .. controls ($(\x+0.1,\y+0.3)$) .. ($(\x+0.2,\y)$);
  }
}
\end{scope}

\draw[waterblue, line width=8pt] (1,-1) to[out=0,in=180] (11.25,-1.75);
\draw[waterblue!30, line width=2pt] (1,-0.8) to[out=0,in=180] (11.25,-1.55);
\draw[waterblue!30, line width=2pt] (1,-1.2) to[out=0,in=180] (11.25,-1.95);

\begin{scope}[shift={(5,-1)}]
\draw[black, line width=1pt] (0,0) circle (0.3);
\draw[black, line width=1pt] (-0.3,0) -- (0.3,0);
\draw[black, line width=1pt] (0,-0.3) -- (0,0.3);
\draw[black, line width=1pt] (-0.21,-0.21) -- (0.21,0.21);
\draw[black, line width=1pt] (-0.21,0.21) -- (0.21,-0.21);
\end{scope}

\draw[waterblue, line width=4pt] (5,-1) -- (2.5,0);
\draw[waterblue, line width=2pt] (5,-1) -- (8,0);

\node[font=\normalsize, align=center] at (6,1.5) {Cardinal\\Non-Comparability};
\node[font=\normalsize, align=center] at (6,0.5) {$ \displaystyle  C(x) = \prod_{i\in \{1,2\}} [\,J_i(x)]^{c_i}$};

\node[font=\normalsize, align=center] at (3.1,-0.6) {$x_1$};
\node[font=\normalsize, align=center] at (6.7,-0.75) {$x_2$};

\node[font=\normalsize, align=center] at (2.5,2.5) {Farmer 1: $J_1(x)$\\ High-value crops, small field};
\node[font=\normalsize, align=center] at (9.5,2.5) {Farmer 2: $J_2(x)$\\Traditional crops, large field};
\end{tikzpicture}}
\caption{Irrigation needs of heterogeneous farmland are difficult to compare.}
\label{fig:irrigation}
\end{figure}

\paragraph*{Comparability}

In irrigation systems, the value a farmer draws from a certain amount of allocated water is difficult to quantify, especially as agricultural farmland and resulting needs are heterogeneous due to crops requiring different amounts of water per hectare and across seasons, with all of it being precipitation dependent~\cite{bwambale2022smart,li2017irrigation}. Thus, one can only conclude that farmers generally experience increased benefits with higher water allocations, though comparing it across heterogeneous farmers is not appropriate~\cite{gomezlimon2020agricultural} (Figure~\ref{fig:irrigation}).
As stated in~\cite{madani2013exogenous}, regulators are interested in designing schemes that robustly satisfy social welfare and justice criteria despite this unmeasurable heterogeneity.
These considerations, in the context of comparability, correspond to Cardinal Non Comparability  \textbf{(CNC)}  of farmers' costs.

\paragraph*{Choice of SCF}
Under \textbf{(CNC)}, the appropriate social welfare function is the Nash Social Welfare. 
For the purpose of water allocation, farmers usually own water rights or water shares to cover the size of their land and account for the crop they are growing. In such a setting, the axiom of Anonymity \textbf{(A)} is not appropriate as the allocation must consider farmers' different access rights to water.
Thus, different exponents $c_i$ are allowed and appropriate in the SCF.

\paragraph*{Implications of the choice of SCF}
Consider a simplified example where every farmer $i$ has water rights $c_i$ and receives a proportion $x^{(i)}$ of the total available water $\bar X$. 
We assume farmers have an upfront cost every season $J_i(x_0)$ and draw marginal utility from an increased water allocation $q_i x^{(i)}$, so that the cost of farmer $i$ is given by $J_i(x) = J_i(x_0) - q_i x^{(i)}$.
The resulting SCF problem becomes 
\begin{align}\label{WaterAllocationProblem}
    \min_{\left\{x^{(i)}\right\}_{i\in\mathcal N}} & \quad -\prod_{i\in\mathcal N} [(q_i x^{(i)}) - J_i(x_0)]^{c_i} \\ \nonumber
    \subjectto  & \quad 0 \le x^{(i)} \quad \forall \,i,  \text{ positive allocation }\\ \nonumber
    & \quad \sum_{i\in\mathcal N} x^{(i)} = \bar X\quad \text{ total volume constraint}.
\end{align}

An interesting case arises when we solve for the optimal social outcome independent of the upfront costs (which the farmers may not disclose); thus, $J_i(x_0)$ (no water, i.e., $x_0^{(i)} = 0 \ \forall i$) acts as a natural worst-case benchmark, corresponding to a planner's choice not to consider upfront expenses. This is reasonable when only water rights should influence the allocation rule. Thus, the farmers themselves need to trade off their costs and marginal benefits with every water share they buy. This approach considers the practical difficulty for a social planner to verify and measure farmers' true expenses or worst-case states.

In such a case, the SCF becomes $-\prod_{i} \left[\,q_i x^{(i)}\right]^{c_i}$.
The solution of~\eqref{WaterAllocationProblem} then satisfies \emph{proportional fairness}~\cite{Kelly1998}, which coincides with the proportional allocation 
\begin{equation*}
x^{(i)*} = \frac{c_i}{\sum_{j\in\mathcal{N}} c_j} \; \bar X
\label{eq:waterproportionalallocation}
\end{equation*}
(see derivation in the appendix of \cite{welfarism-arxiv}).

Proportional allocation is a commonly used allocation procedure 
implemented in constituencies worldwide~\cite{ostrom1994rules, roagarcia2014equity,gomezlimon2020agricultural, rosegrant1994markets}.
Thus, using social choice theory arguments and analyzing the underlying notion of comparability, we can offer a different perspective on the commonly used proportional allocation rule in water irrigation systems.

\subsection{Traffic Control}
\label{subsec:traffic}

Fig.~\ref{fig:routing} shows an example of the commonly studied traffic routing problem~\cite{Chandan2024,Zhang2018,Piliouras2017,Wang2017} with $40$ \emph{long-distance commuters} travelling from $O_1$ through $O_2$ to $D$, and $40$ \emph{short-distance commuters} travelling from $O_2$ to $D$ only.
The link delays are as shown in Fig.~\ref{fig:network}.
We next discuss whether travel delay costs of both commuter types can be compared and illustrate the consequences of different comparability assumptions.

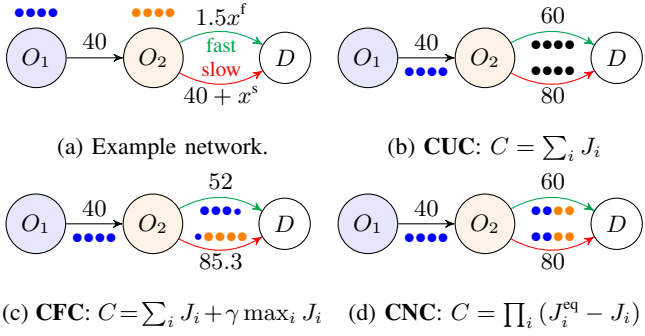
\begin{figure}[b!]
    \centering
    \begin{subfigure}[t]{0.24\textwidth}
        \centering
                \begin{tikzpicture}[
            node distance = 2.5cm,
            ->,
            >=stealth',
            font=\small,
        ]
            \node[state, minimum size = 0pt, fill=blue!10] (O1) {$O_1$};
            \node[state, minimum size = 0pt, right=0.75cm of O1, fill=orange!10] (O2) {$O_2$};
            \node[state, minimum size = 0pt, right=1cm of O2] (D) {$D$};

            \node[draw=none, above=0cm of O1] {\footnotesize \color{blue} $\bullet$$\bullet$$\bullet$$\bullet$};
            \node[draw=none, above=0cm of O2] {\footnotesize \color{orange} $\bullet$$\bullet$$\bullet$$\bullet$};
            
            \draw (O1) edge node[above]{$40$} (O2);
            \draw (O2) edge[bend left, draw=Green] node{\begin{tabular}{c} $1.5 x^\textup f$ \\[0.1mm] \footnotesize \color{Green}fast \end{tabular}} (D);
            \draw (O2) edge[bend right, draw=red] node{\begin{tabular}{c} \footnotesize \color{red}slow \\[0.1mm] $40 + x^\textup s$\end{tabular}} (D);
        \end{tikzpicture}
        \caption{Example network.}
        \label{fig:network}
    \end{subfigure}
    \hfill
    \begin{subfigure}[t]{0.235\textwidth}
        \centering
                \begin{tikzpicture}[
            node distance = 2.5cm,
            ->,
            >=stealth',
            font=\small,
        ]
            \node[state, minimum size = 0pt, fill=blue!10] (O1) {$O_1$};
            \node[state, minimum size = 0pt, right=0.75cm of O1, fill=orange!10] (O2) {$O_2$};
            \node[state, minimum size = 0pt, right=1cm of O2] (D) {$D$};
            
            \draw (O1) edge node{\begin{tabular}{c} $40$ \\[0.1mm] \footnotesize \color{blue} $\bullet$$\bullet$$\bullet$$\bullet$ \end{tabular}} (O2);
            \draw (O2) edge[bend left, draw=Green] node{\begin{tabular}{c} $60$ \\[0.1mm] \footnotesize $\bullet$$\bullet$$\bullet$$\bullet$ \end{tabular}} (D);
            \draw (O2) edge[bend right, draw=red] node{\begin{tabular}{c} \footnotesize $\bullet$$\bullet$$\bullet$$\bullet$ \\[0.1mm] $80$ \end{tabular}} (D);
        \end{tikzpicture}
        \caption{\textbf{CUC}: $C = \sum_i J_i$}
        \label{fig:route-CUC}
    \end{subfigure}
    
    \begin{subfigure}[t]{0.24\textwidth}
        \centering
                \begin{tikzpicture}[
            node distance = 2.5cm,
            ->,
            >=stealth',
            font=\small,
        ]
            \node[state, minimum size = 0pt, fill=blue!10] (O1) {$O_1$};
            \node[state, minimum size = 0pt, right=0.75cm of O1, fill=orange!10] (O2) {$O_2$};
            \node[state, minimum size = 0pt, right=1cm of O2] (D) {$D$};
            
            \draw (O1) edge node{\begin{tabular}{c} $40$ \\[0.1mm] \footnotesize \color{blue} $\bullet$$\bullet$$\bullet$$\bullet$ \end{tabular}} (O2);
            \draw (O2) edge[bend left, draw=Green] node{\begin{tabular}{c} $52$ \\[0.1mm] \footnotesize \color{blue} $\bullet$$\bullet$$\bullet$\raisebox{0.24mm}{\tiny$\bullet$} \end{tabular}} (D);
            \draw (O2) edge[bend right, draw=red] node{\begin{tabular}{c} \raisebox{0.24mm}{\tiny\color{blue}$\bullet$}\footnotesize\color{orange}$\bullet$$\bullet$$\bullet$$\bullet$ \\[0.1mm] $85.3$ \end{tabular}} (D);
        \end{tikzpicture}
        \caption{\textbf{CFC}: $C\!=\!\sum_i J_i + \gamma \max_i J_i$}
        \label{fig:route-CFC}
    \end{subfigure}
    \hfill
    \begin{subfigure}[t]{0.235\textwidth}
        \centering
                \begin{tikzpicture}[
            node distance = 2.5cm,
            ->,
            >=stealth',
            font=\small,
        ]
            \node[state, minimum size = 0pt, fill=blue!10] (O1) {$O_1$};
            \node[state, minimum size = 0pt, right=0.75cm of O1, fill=orange!10] (O2) {$O_2$};
            \node[state, minimum size = 0pt, right=1cm of O2] (D) {$D$};
            
            \draw (O1) edge node{\begin{tabular}{c} $40$ \\[0.1mm] \footnotesize \color{blue} $\bullet$$\bullet$$\bullet$$\bullet$ \end{tabular}} (O2);
            \draw (O2) edge[bend left, draw=Green] node{\begin{tabular}{c} $60$ \\[0.1mm] \footnotesize \color{blue} $\bullet$$\bullet$\color{orange}$\bullet$$\bullet$ \end{tabular}} (D);
            \draw (O2) edge[bend right, draw=red] node{\begin{tabular}{c} \footnotesize\color{blue}$\bullet$$\bullet$\color{orange}$\bullet$$\bullet$ \\[0.1mm] $80$ \end{tabular}} (D);
        \end{tikzpicture}
        \caption{\textbf{CNC}: $C = \prod_i \left(J_i^\textup{eq} - J_i\right)$}
        \label{fig:route-CNC}
    \end{subfigure}
    \caption{Traffic example illustrating different comparability assumptions and the associated optimal outcomes.}
    \label{fig:routing}
\end{figure}

\paragraph*{Comparability}
The prevalence of the sum of costs as SCF in the literature suggests that either \textbf{(CUC)} or \textbf{(CFC)} is implicitly assumed.
\textbf{(CUC)} is often assumed in conjunction with tolling solutions~\cite{yang1998principle}, in which the unit of comparison is the monetary value of time~\cite{zamparini2007meta}, and the commuters' dispensable incomes, i.e., the affine offsets $b_i$, are not considered in the cost functions.
Works that instead explicitly account for income differences\cite{jalota2021efficiency} can be classified under \textbf{(CFC)}.
\textbf{(CFC)} is also natural when considering the travel delays themselves as costs (and not how they convert to money), with the justification that everyone has 24 hours in a day.

One could instead assume that travel delays are non-comparable \textbf{(CNC)}, justified as follows.
The commuter heterogeneity is attributed to non-comparable trade-offs, e.g.,
long-distance commuters prioritize large housing over short commutes, and vice versa for short-distance commuters.

\paragraph*{Choice of SCF}
In addition to basic axioms \textbf{(P)}, \textbf{(IIA)}, Anonymity \textbf{(A)} is desired if commuters are to have equal access rights to roads. Therefore, under \textbf{(CUC)} the only appropriate SCF is $\sum_i J_i$, while a larger family of SCFs is permissible under \textbf{(CFC)} (see Table~\ref{table}).
To balance a trade-off between total delays and delays of the worst off, one can choose SCF $\sum_i J_i + \gamma \max J_i$.
In our example, $\gamma$ essentially dictates to what extent long-distance commuters should be compensated in $O_2 \rightarrow D$ for their additional delay in $O_1 \rightarrow O_2$.
Under \textbf{(CNC)}, since travel delay costs are nonnegative, one must resort to \emph{Partial Independence} \textbf{(PI)}, and choose a suitable reference outcome $x_0$.
A natural choice of $x_0$ is the non-controlled traffic equilibrium, since it represents the \emph{status quo} or \emph{disagreement point} if negotiations to adopt new control policies fail, and is commonly adopted in transportation to certify that policies are Pareto improving~\cite{daganzo2000pareto,elokda2024carma}.
This leads to SCF $\prod_i \left(J_i^\textup{eq} - J_i\right)$, with $J_i^\textup{eq}$ denoting the equilibrium delay to commuter $i$.

\paragraph*{Implications of the choice of SCF}
The SCF can be used in an optimization formulation to determine the socially optimal traffic outcome.
Figs.~\ref{fig:route-CUC}--\ref{fig:route-CNC} illustrate the results in our example network, indicating the optimal link delays and associated allocations of long-distance (blue) and short-distance (orange) commuters\footnote{The allocations are to be interpreted in a frequentist sense, e.g., in Fig.~\ref{fig:route-CNC} each commuter uses the fast/slow link half of the times.}.
Notice that under \textbf{(CUC)}, it does not matter which commuter uses the fast/slow link (indicated black in Fig.~\ref{fig:route-CUC}).
While these different allocations may appear intuitively more or less fair or efficient, we emphasize that the appropriate notion of ``fairness and efficiency'' is underpinned by the assumed and justifiable level of comparability.

\subsection{Energy Curtailment}
\label{subsec:energy}

Electrical power distribution grids host an ever-increasing amount of renewable power generation (e.g., residential solar panels). 
These grids have finite power transfer capacity, dictated by the physical limits of the infrastructure, which is expensive and sometimes impossible to reinforce.
Therefore, access needs to be regulated: when overproduction occurs and power needs to be exported from these generators to the rest of the grid, some of the generation needs to be curtailed (up to 10\% of the power generated by new installations \cite{NOVAN2024102930}).

In mathematical terms, for a pool of $N$ generators, the grid operators need to decide the power curtailments $\{x_i\}_{i=1}^N$ for each generator $i$. In deciding that, they need to ensure that the grid operational constraints (voltage and line current limits) are satisfied when generators produce $p_i - x_i$, where $p_i$ is their current potential production.
Such a feasibility problem leaves the grid operators to decide what is the ``best'' curtailment: on the one hand, the transition towards a sustainable power system calls for maximizing the total generation from renewable sources (in some cases, by law \cite{EEG2017}). On the other hand, the energy system is a shared infrastructure, and each generator expects fair access to it. See \cite{soares2024review} for a review of fairness in energy systems and \cite{DallAnese2014,gebbran2021fair, lusis2019reducing, liu2020fairness, Borbath2024, Moring2024Fair-Over-TimeCoordination} for concrete proposals for fair curtailment strategies. These works propose multiple solutions but provide little guidance about which one to select, except for an \emph{a posteriori} quantitative assessment of various fairness indices, which are equally hard to choose and justify.

\paragraph*{Comparability}

Different comparability assumptions are possible. We give some examples.

One could argue that each MW of curtailed power is fundamentally comparable and identical (for example, because the generators are in the same price zone and incur the same financial cost), implying $J_i(x) = x_i$ and \textbf{(CUC)}.

Alternatively, one can acknowledge the heterogeneous nature of these small stakeholders (residential users, solar farms, etc.) and conclude that no cardinal comparability of the financial cost is possible. However, it is possible to postulate when two generators are treated equitably, and assume \textbf{(OLC)}.
For example, with $J_i(x) = x_i$, \textbf{(OLC)} means that two generators are treated equally if the same amount of power is curtailed, as they incur the same financial cost.
If instead $J_i(x) = x_i/p_i$, where $p_i$ is the power that generator $i$ would be able to produce, then \textbf{(OLC)} means that generators are treated equally if the same fraction of power is curtailed, as they incur the same financial cost normalized to their earning opportunity and their investments.
Finally, if $J_i(x) = x_i - p_i$, then \textbf{(OLC)} means that generators exporting (i.e., selling) the same amount of power to the grid are treated equally, as they are using the grid equally.

It is a duty of the designer of the curtailment policy to interpret the mandate they received from the stakeholders and from the local regulations.
Arguably, it is easier to extract these sentiments by comparing and selecting a comparability notion than by evaluating curtailment policies or analyzing their effect once deployed.

\paragraph*{Choice of SCF}

The aforementioned comparability levels dictate which SCF is appropriate to use. \textbf{(CUC)} leaves no other choice than the utilitarian approach of maximizing $\sum_i x_i$ (which corresponds to \emph{OPF-total} in \cite{liu2020fairness}).
Under \textbf{(OLC)}, the appropriate SCF is necessarily $\max_i J_i(x)$. 
When $J_i(x) = x_i$, we recover the \emph{Egalitarian} curtailment policy proposed in \cite{gebbran2021fair}. When $J_i(x) = x_i/p_i$, we recover the \emph{OPF-generation} policy in \cite{liu2020fairness} (also similar to the \emph{Proportional} policy in \cite{gebbran2021fair}). Finally, when $J_i(x) = x_i - p_i$, we recover the \emph{Uniform Dynamic} policy in \cite{gebbran2021fair} and the \emph{OPF-export} policy in \cite{liu2020fairness}.

\paragraph*{Implications of the choice of SCF}

Different comparability assumptions justify different policies that have been proposed in the literature, but have not been adequately motivated. 
Explainability is crucial in this application, as the financial consequences are significant (e.g., the utilitarian approach results in a ``water filling'' solution where generators are completely curtailed starting from those connected further away from the main grid).
Figure~\ref{fig:curtailment-policies} provides some intuition on the different curtailment policies. We refer to \cite{gebbran2021fair} for a comparison of the numerical solutions to the resulting optimization problems on a benchmark distribution grid.

Inequitable and opaque access to the electricity grid exacerbate existing social disparities \cite{Brockway2021}, hinders investments \cite{Cuenca2023}, and may not meet the legal mandate of DSO as transparent and neutral players \cite{EU2019_944,CERRE2022_DSO}.

\begin{figure}[tb!]
    \centering
    \includegraphics[width=\linewidth]{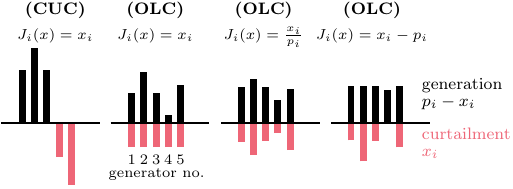}
    \caption{The effect of different curtailment policies (justified by different comparability assumptions) on five generators. Curtailing generators 4 and 5 is the most efficient way to satisfy the grid operational constraints. Figure adapted from \cite{gebbran2021fair}.}
    \label{fig:curtailment-policies}
\end{figure}

Besides providing a principled motivation for different curtailment policies, the proposed approach warns against using fairness indices like those reviewed in \cite{soares2024review, lusis2019reducing} (e.g., Gini index, Jain index, F-fairness, K-fairness). While they are useful to measure inequality, they may have unintended consequences when used as social cost functions, as they violate basic axioms like \emph{Pareto optimality} of the solution.

\section{Conclusions and Outlook}
\label{sec:conclusions}

Much of applied control work has implicitly been performed with the interests of the population in mind and based on ``welfarist'' principles as we laid out in this paper. We aimed to provide guidance for making the foundations of the use of social cost functions explicit and, thus, to enable control theorists to express and defend their objectives based on the limits/possibilities of interpersonal comparability in various contexts and applications.
In some cases, some social cost functions have to be ruled out, while many functions are permissible in other situations, and our theory thus may restrict or make explicit the range of control objectives that an engineer may pursue. 
In practice, the designer should first state the admissible level of interpersonal comparability, then select a social cost function from the class that respects this choice, and report both as part of the problem specification. This turns an implicit convention into an explicit specification and allows one to understand and examine the fairness and efficiency trade-offs.

Multiple questions remain open and should be investigated to enable an effective use of this social-choice-theory perspective in control design problems. We list three directions.
\begin{itemize}
    \item \textbf{Closed-loop certificates} -- 
    The examples presented in Section~\ref{sec:examples} are static decision problems. In a control problem, one has to design a feedback policy and then analyze its behavior in the closed-loop interconnection with the plant, taking into account stochastic disturbances, parametric uncertainty, and measurement noise. Closed-loop certificates for the transient and tracking performance need to be derived, employing operations like those proposed in Section~\ref{ssec:usesofW}.
    \item \textbf{Online SCF optimization} -- In dynamic settings, exogenous information is incrementally revealed to the decision maker, agents' preferences are time-varying, and present decisions affect future trajectories. Social choice results need to be combined with computationally tractable dynamic programming tools to design principled online policies \cite{Kulkarni2020ChangingPreferences,Freeman2017DynamicSettings}.
    \item \textbf{Intermediate comparability notions} -- In realistic settings, the level of comparability between agents' costs may be revealed through noisy measurements, prior beliefs, and users' feedback. The standard classification offered by social choice theory is too coarse to capture these settings, and the theory needs to be developed to allow principled social decisions in the face of more nuanced comparability levels.    
\end{itemize}

\bibliographystyle{IEEEtran}
\bibliography{bibliography,bib_MOO}


\end{document}